\documentclass[review]{elsarticle}
\usepackage{lineno,hyperref}
\usepackage{array}
\usepackage[all]{xy}
\usepackage{color}
\usepackage{amsxtra, microtype}
\usepackage{amssymb,amsmath,amsthm,stmaryrd,latexsym,wasysym}
\usepackage{amstext}
\usepackage{dsfont}
\usepackage{graphicx}
\usepackage{leftidx}
\usepackage[version=3]{mhchem}

\newtheorem{theorem}{Theorem}[section]

\newtheorem{corollary}[theorem]{Corollary}

\newtheorem{definition}[theorem]{Definition}
\newtheorem{example}[theorem]{Example}

\newtheorem{lemma}[theorem]{Lemma}

\newtheorem{proposition}[theorem]{Proposition}
\newtheorem{remark}[theorem]{Remark}

\numberwithin{equation}{section}
\newcommand{\duer}{\mathbin{\raisebox{3pt}{\varhexstar}\kern-3.70pt{\rule{0.15pt}{4pt}}}\,}
\modulolinenumbers[5]

\journal{...}

\makeatletter
\def\ps@pprintTitle{%
   \let\@oddhead\@empty
   \let\@evenhead\@empty
   \def\@oddfoot{\reset@font\hfil\thepage\hfil}
   \let\@evenfoot\@oddfoot
}
\makeatother
\biboptions{sort&compress}

\begin{document}

\begin{frontmatter}

\title{$\mathds{Z}$-graded Hom-Lie Superalgebras}

\author{Mohammad Reza Farhangdoost\corref{mycorrespondingauthor}}
\cortext[mycorrespondingauthor]{Corresponding author:\quad farhang@shirazu.ac.ir}

\author{Ahmad Reza Attari Polsangi\corref{}}

\address{Department of Mathematics, College of Sciences,\\Shiraz University, P.O. Box 71457-
44776, Shiraz, Iran}

\begin{abstract}
 In this paper we introduce the notions of $\mathds{Z}$-graded hom-Lie superalgebras and we show that there is a maximal (resp., minimal) $\mathds{Z}$-graded hom-Lie superalgebra for a given local hom-Lie superalgebra. Morever, we introduce the invariant bilinear forms on a $\mathds{Z}$-graded hom-Lie superalgebra and we prove that a consistent supersymmetric $\alpha$-invariant form on the local part can be extended uniquely to a bilinear form with the same property on the whole $\mathds{Z}$-graded hom-Lie superalgebra. Furthermore, we check the condition in which the $\mathds{Z}$-graded hom-Lie superalgebra is simple.
\end{abstract}

\end{frontmatter}

\section{Introduction}
The notion of hom-Lie algebras was introduced by J. T. Hartwig, D. Larsson and S. Silvestrov in \cite{hartwig2006deformations} described the structures on certain deformations of Witt algebras and Virasoro algebras, which are widely utilized in the theoretical physics; such as string theory, vertex models in conformal field theory, quantum mechanics and quantum field theory \cite{aizawa1991q, chaichian1990quantum, chaichian1991q, daskaloyannis1992generalized, liu1992characterizations}. The hom-Lie algebras attract more and more attention. New structures on hom-Lie algebras which have great importance and utility, have been defined and discussed by mathematicians that we express them: The hom-Lie algebras were discussed by D. Larsson and S. Silvestrov in \cite{larsson2005quasi-hom, larsson2005quasi, larsson2007quasi}. The authors studied quadratic hom-Lie algebras in \cite{benayadi2014hom}; representation theory, cohomology and homology theory in \cite{ammar2010cohomology, sheng2012representations, yau2009hom}. In \cite{larsson2005quasi, larsson2005graded, larsson2009generalized, sigurdsson2006graded}, S. Silvestrov et al. introduced the general quasi-Lie algebras and including as special cases the color hom-Lie algebras \cite{armakan2017extensions, armakan2019extensions, armakan2019enveloping} and in particular hom-Lie superalgebras. Recently, different features of hom-Lie superalgebras has been studied by authors in \cite{ammar2013cohomology, ammar2010hom, armakan2017geometric, zhou2014general, makhlouf2010paradigm}. On the other hand, we have the graded structure on Lie algebras which were discussed by authors. The first basic example of graded Lie algebras was provided by Nijenhuis and then by Frolicher and Nijenhuis in \cite{frolicher1957theorem}. In \cite{kac1968simple}, Kac introduced some features of graded Lie algebras and then generalized it to Lie superalgebras in \cite{kac1977lie}. Now we want to construct a special graded structure ($\mathds{Z}$-graded), on hom-Lie superalgebras.
In the first section of this paper, hom-Lie algebras, hom-Lie superalgebras and some of their useful related definitions are presented. In section two, we present the notion of $\mathds{Z}$-graded hom-Lie superalgebras, local hom-Lie superalgebras and we prove that for a given multiplicative local hom-Lie superalgebra such that ${\alpha}^2 = \alpha$, there is a maximal and a minimal $\mathds{Z}$-graded hom-Lie superalgebra in which local parts are isomorphic to it. In the last section, we focus on invariant bilinear forms and we show that we can extended uniquely a consistent supersymmetric $\alpha$-invariant bilinear form on the local part to the whole $\mathds{Z}$-graded hom-Lie superalgebra with the same property.\\
Throughout this paper we fix a ground field $\mathds{K}$. All $\mathds{Z}_{2}$-graded vector spaces are considered over $\mathds{K}$ and linear maps are $\mathds{K}$-linear maps. Each element in the hom-Lie superalgebra is supposed to be homogeneous and degree of $x$ is denoted by $|x|$, for all $x \in \mathfrak{g}$.\\

We recall the definition of hom-Lie algebras and Lie superalgebras from \cite{sheng2012representations, kac1977lie}. Also, we present notions of a hom-Lie superalgebra as a generalization of a Lie superalgebra, \cite{makhlouf2010paradigm}.

\begin{definition}\cite{sheng2012representations}
	A hom-Lie algebra is a triple $(\mathfrak{g},[.,.],\alpha)$, where $\mathfrak{g}$ is a vector space equipped with a skew-symmetric bilinear map
	$[.,.] : \mathfrak{g} \times  \mathfrak{g} \to$  $\mathfrak{g}$
	and a linear map
	$\alpha : \mathfrak{g} \to \mathfrak{g}$
	such that;
	\begin{equation*}
		[\alpha (x), [y, z]] + [\alpha (y), [z, x]] + [\alpha (z), [x, y]] = 0,
	\end{equation*}
	for all $x$, $y$, $z$ in $\mathfrak{g}$.
\begin{itemize}
	\item 	A hom-Lie algebra is called a multiplicative hom-Lie algebra, if $\alpha$ is an algebraic morphism, i.e. for any $x, y \in \mathfrak{g}$,
	\begin{equation*}
		\alpha([x,y])=[\alpha(x),\alpha(y)].
	\end{equation*}
	\item 	We call a hom-Lie algebra regular, if $\alpha$ is an automorphism.
\end{itemize}
\end{definition}

\begin{definition} \cite{kac1977lie}
	A Lie superalgebra is a $\mathds{Z}_{2}$-graded vector space $\mathfrak{g} = \mathfrak{g}_{\bar{0}} \oplus \mathfrak{g}_{\bar{1}}$, together with a graded Lie bracket $[.,.] : \mathfrak{g} \times \mathfrak{g} \to \mathfrak{g}$ of degree zero, i.e. $[.,.]$ is a bilinear map with
	\begin{equation*}
		[\mathfrak{g}_{i}, \mathfrak{g}_{j}] \subset \mathfrak{g}_{{i+j}(mod2)},
	\end{equation*}
	such that for homogeneous elements $x,y,z \in \mathfrak{g}$, the following identities hold:
	\begin{itemize}
		\item $[x,y]=-(-1)^{|x||y|}[y,x]$,
		\item $[x,[y,z]]=[[x,y],z]+(-1)^{|x||y|}[y,[x,z]]$.
	\end{itemize}
\end{definition}

Now analogous to the definition of hom-Lie algebras, a hom-Lie superalgebra is defined in a way that makes it a generalization of a Lie superalgebra.

\begin{definition}\label{HLS} \cite{makhlouf2010paradigm}
A hom-Lie superalgebra is a triple $(\mathfrak{g},[.,.],\alpha)$ consisting of a $\mathds{Z}_{2}$-graded vector space $\mathfrak{g} = \mathfrak{g}_{\bar{0}} \oplus \mathfrak{g}_{\bar{1}}$, an even linear map (bracket) $[.,.] : \mathfrak{g} \times \mathfrak{g} \to \mathfrak{g}$ and an even homomorphism $\alpha : \mathfrak{g} \to \mathfrak{g}$ satisfying the following supersymmetry and hom-Jacobi identity, i.e.
\begin{itemize}
	\item  $[x,y]=-(-1)^{|x||y|}[y,x]$,
	\item $(-1)^{|x||z|}[\alpha(x),[y,z]]+(-1)^{|y||x|}[\alpha(y),[z,x]]+(-1)^{|z||y|}[\alpha(z),[x,y]]=0$,
where $x$, $y$ and $z$ are homogeneous elements in $\mathfrak{g}$.
\end{itemize}
\begin{itemize}
	\item A hom-Lie superalgebra is called multiplicative hom-Lie superalgebra, if $\alpha$ is an algebraic morphism, i.e. for any $x, y \in \mathfrak{g}$ we have
	\begin{equation*}
		\alpha([x,y])=[\alpha(x),\alpha(y)].
	\end{equation*}
	\item A hom-Lie superalgebra is called regular hom-Lie superalgebra, if $\alpha$ is an algebraic automorphism.
\end{itemize}
\end{definition}

\begin{remark} \cite{ammar2010hom}
    Setting $\alpha = id$ in Definition \ref{HLS} we obtain the definition of a Lie superalgebra. Hence hom-Lie superalgebras include Lie superalgebras as a subcategory, thereby motivating the name ”hom-Lie superalgebras” as a deformation of Lie superalgebras by an endomorphism.
\end{remark}

\begin{example} \cite{ammar2010hom}
	(Affine hom-Lie superalgebra). Let $V=V_0 \oplus V_1$ be a 3-dimensional superspace where $V_0$ is generated by $e_1 , e_2$ and $V_1$ is generated by $e_3$. The triple $(V,[.,.],\alpha)$ is a hom-Lie superalgebra defined by $[e_1,e_2]=e_1$, $[e_1,e_3]=[e_2,e_3]=[e_3,3_3]=0$ and $\alpha$ is any homomorphism.
\end{example}

Let $(\mathfrak{g},[.,.],\alpha)$ and $(\mathfrak{g}',[.,.]',\beta)$ be two hom-Lie superalgebras. An even homomorphism $ \phi : \mathfrak{g} \to \mathfrak{g}'$ is said to be a homomorphism of hom-Lie superalgebras, if
	
	$ \phi [u,v]=[\phi (u), \phi (v))]'$,
	
	$\phi o \alpha = \beta o \phi$.

The hom-Lie superalgebras $(\mathfrak{g},[.,.],\alpha)$ and $(\mathfrak{g}',[.,.]',\beta)$ are isomorphic, if there is a hom-Lie superalgebra homomorphism $ \phi : \mathfrak{g} \to \mathfrak{g}'$ such that $\phi$ be bijective \cite{makhlouf2010paradigm}.

A sub-vector space $I \subseteq \mathfrak{g}$ is a hom-subalgebra of $(\mathfrak{g},[.,.],\alpha)$, if $\alpha(I) \subseteq I$ and I is closed under the bracket operation $[.,.]$, i.e. $[I,I] \subseteq I$. Also, hom-subalgebra I is called a hom-ideal of $\mathfrak{g}$, if $[I,\mathfrak{g}] \subseteq I$. Moreover, if $[I,I] = 0$, then I is Abelian \cite{guan2019on}.

We can now make some new hom-ideal that is stated in the following lemma.

\begin{lemma}
	Let $I_{1}, I_{2}$ be hom-ideal in $\mathfrak{g}$. Define
	
	$ I_{1} + I_{2} = \{x+y | x \in I_{1} , y \in I_{2}\}$,

	$ [I_{1},I_{2}]=$
	Subspace spanned by
	$[x,y], x \in I_{1}, y \in I_{2}$.\\ Then $ I_{1} \cap I_{2}, I_{1}+I_{2}, [I_{1}, I_{2}]$ are hom-ideal in $\mathfrak{g}$.
\end{lemma}

\begin{proof}
	The proof that $I_{1} \cap I_{2}$ and $I_{1}+I_{2}$ are hom-ideal is straightforward. So we just prove that $[I_{1}, I_{2}]$ is hom-ideal. Let $t \in \mathfrak{g}$ and $x \in [I,J]$. Thus $x$ is written as the form $$x=c_1[x_1,y_1]+...+c_n[x_n,y_n],$$ where $c_i \in \mathds{K}$, $x_i \in I$ and $y_i \in J$, for $i = 1, 2, ..., n$.\\ Without loss of generality, assumes that $x=c[x,y]$. So
		\begin{align*}
			[\alpha(t),x] &= [\alpha(t), c[x,y]]\\ &= c[\alpha(t), [x,y]] \\ &=c((-1)^{|x||t|}[\alpha(x), [t,y]] + (-1)^{|y||x|}[[t,x],\alpha(y)]) \in I+J.
		\end{align*}
		Finally
		\begin{align*}
		\alpha(x) = \alpha(c[x,y]) = c(\alpha[x,y]) = c([\alpha(x), \alpha(y)]) \in I+J.
		\end{align*}
\end{proof}

\begin{definition} \cite{liu2013hom}
	The center of a hom-Lie superalgebra $\mathfrak{g}$, denoted by $Z(\mathfrak{g})$, is the set of elements $x \in \mathfrak{g}$ satisfying $[x,\mathfrak{g}]=0$.
\end{definition}

\begin{remark}
When $\alpha : \mathfrak{g} \to \mathfrak{g}$ is a surjective endomorphism, then one can easily check that $(Z(\mathfrak{g}), \alpha)$
is an Abelian hom-Lie superalgebra and an hom-ideal of $(\mathfrak{g},[.,.],\alpha)$.
\end{remark}

\begin{lemma}
	The quotient $\mathfrak{g}/[\mathfrak{g},\mathfrak{g}]$ is an Abelian hom-Lie superalgebra. Moreover, $[\mathfrak{g},\mathfrak{g}]$ is the smallest hom-ideal with this property: if $\mathfrak{g}/I$ is Abelian for some hom-ideal $ I \subset \mathfrak{g}$, then $[\mathfrak{g},\mathfrak{g}] \subset I$.
\end{lemma}

\begin{proof}
    For $x \in \mathfrak{g}$, $\bar{x}$ denotes its image in $\mathfrak{g}/I$ and the Lie bracket on $\mathfrak{g}/I$ is defined by setting
    $[\bar{x},\bar{y}]= \overline{[x,y]}$. Let $\bar{x}, \bar{y} \in \mathfrak{g}/[\mathfrak{g},\mathfrak{g}]$, then
    \begin{equation*}
      [\bar{x},\bar{y}]= \overline{[x,y]}=0 \in \mathfrak{g}/[\mathfrak{g},\mathfrak{g}].
    \end{equation*}
    So $\mathfrak{g}/[\mathfrak{g},\mathfrak{g}]$ is an Abelian hom-Lie superalgebra. Moreover, if $\mathfrak{g}/I$ is Abelian, $\bar{x}$ and $\bar{y}$ commute, then
    \begin{equation*}
      \overline{[x,y]} = [\bar{x},\bar{y}]=0 \in \mathfrak{g}/I.
    \end{equation*}
    This implies that $[x,y] \in I$ and therefore we have $[\mathfrak{g},\mathfrak{g}] \subset I$.
\end{proof}

We are going to need the following definition throughout the rest of the paper.

\begin{definition} \cite{ammar2013cohomology}
	A representation of the hom-Lie superalgebra $(\mathfrak{g},[.,.],\alpha)$ on a $\mathds{Z}_{2}$-graded vector space $V = V_{\bar{0}} \oplus V_{\bar{1}} $ with respect to $\beta \in gl(V)_{\bar{0}}$ is an even linear map $\rho : \mathfrak{g} \to gl(V)$, such that for all $x,y \in \mathfrak{g}$, the following equalities are satisfied:
	\begin{equation*}
	\rho (\alpha(x)) o \beta = \beta o \rho(x);
	\end{equation*}
	\begin{equation*}
	\rho ([x,y]) o \beta = \rho (\alpha(x)) o \rho(y) -(-1)^{|x||y|} \rho (\alpha(y)) o \rho(x).
	\end{equation*}

A representation $V$ of $\mathfrak{g}$ is called irreducible or simple, if it has no nontrivial subrepresentations. Otherwise $V$ is called reducible.
\end{definition}

Let $(\mathfrak{g},[.,.],\alpha)$ be a multiplicative hom-Lie superalgebra. We consider $\mathfrak{g}$ as a representation on itself via the bracket and with respect to the morphism $\alpha$.

\begin{example} \cite{ammar2013cohomology}
The ${\alpha}^s$-adjoint representation of the hom-Lie superalgebra $(\mathfrak{g},[.,.],\alpha)$, which we denote by $ad_s$, is defined by
\begin{equation*}
  ad_s(a)(x)= [{\alpha}^s(a),x], \quad \text{for all} \quad a, x \in \mathfrak{g}.
\end{equation*}
\end{example}

\begin{lemma} \cite{ammar2013cohomology}
With the above notation, we have that $(\mathfrak{g},ad_s(.)(.), \alpha)$ is a representation of the hom-Lie superalgebra $\mathfrak{g}$.
\end{lemma}
\section{$\mathds{Z}$-graded and local hom-Lie superalgebras}

In \cite{kac1977lie}, Kac introduced the notion of a $\mathds{Z}$-graded Lie superalgebra and in this section we introduce concept of a $\mathds{Z}$-graded hom-Lie superalgebra and state some results about it.

\begin{definition}
Let $\mathfrak{g}$ be a hom-Lie superalgebra. It is called a $\mathds{Z}$-graded hom-Lie superalgebra if it is decomposision of itself into a direct sum of finite-demensional $\mathds{Z}_2$-graded subspaces $\mathfrak{g}= \bigoplus_{i \in \mathds{Z}} \mathfrak{g}_i$, for which $[\mathfrak{g}_i , \mathfrak{g}_j] \subseteq \mathfrak{g}_{i+j}$.\\

A $\mathds{Z}$-graded hom-Lie superalgebra is said to be consistent, if
$\mathfrak{g}_{\bar{0}} = \bigoplus_{i \in \mathds{Z}} \mathfrak{g}_{2i}$ and $\mathfrak{g}_{\bar{1}}= \bigoplus_{i \in \mathds{Z}} \mathfrak{g}_{2{i+1}}$.
\end{definition}

\begin{remark}
	One can see that if $\mathfrak{g}$ is a $\mathds{Z}$-graded hom-Lie superalgebra, then $\mathfrak{g}_0$ is a hom-Lie subalgebra and $[\mathfrak{g}_{0} , \mathfrak{g}_j] \subseteq \mathfrak{g}_j$; therefore the restriction of the adjoint representation to $\mathfrak{g}_{0}$ induces linear representation of it on the subspaces $\mathfrak{g}_j$.
\end{remark}

\begin{definition}
	We have some basic definitions as follows:
	\begin{itemize}
		\item A $\mathds{Z}$-graded hom-Lie superalgebra $\mathfrak{g}= \bigoplus_{i \in \mathds{Z}} \mathfrak{g}_i$ is called irreducible, if the representation of $\mathfrak{g}_0$ on $\mathfrak{g}_{-1}$ is irreducible.
		\item A $\mathds{Z}$-graded hom-Lie superalgebra $\mathfrak{g}= \bigoplus_{i \in \mathds{Z}} \mathfrak{g}_i$ is called transitive, if for $a \in \mathfrak{g}_i$, $i \ge 0$, it follows from $[a,\mathfrak{g}_{-1}]=0$ that $a=0$.
		\item A $\mathds{Z}$-graded hom-Lie superalgebra $\mathfrak{g}= \bigoplus_{i \in \mathds{Z}} \mathfrak{g}_i$ is called bitransitive, if in addition for $a \in \mathfrak{g}_i$, $i \le 0$, it follows from $[a,\mathfrak{g}_{1}]=0$ that $a=0$.
	\end{itemize}
\end{definition}

\begin{definition} \label{simple}
	A $\mathds{Z}$-graded hom-Lie superalgebra $\mathfrak{g}$ is called simple, if it does not have any nontrivial graded hom-ideal and $[\mathfrak{g},\mathfrak{g}] \ne \{0\}$.
\end{definition}

\begin{remark}
	From Definition \ref{simple}, one can consider the followings:
	\begin{itemize}
		\item A left or right graded hom-ideal of $\mathfrak{g}$ is automatically a two sided hom-ideal.
		\item The condition $[\mathfrak{g},\mathfrak{g}] \ne \{0\}$ serves to eliminate the zero-dimensional and two one-dimensional hom-Lie superalgebras. It follows that $[\mathfrak{g},\mathfrak{g}] = \mathfrak{g}$.
	\end{itemize}
\end{remark}

These properties are closely connected with $\mathfrak{g}$ being simple, as is shown in the following proposition.

\begin{proposition}\label{prop1}
	Let $\mathfrak{g}= \bigoplus_{i \in \mathds{Z}} \mathfrak{g}_i$ be a simple $\mathds{Z}$-graded hom-Lie superalgebra. If the subspace $\mathfrak{g}_{-1} \oplus \mathfrak{g}_0 \oplus \mathfrak{g}_1$ generates $\mathfrak{g}$, then it is bitransitive.
\end{proposition}

\begin{proof}
	If, we consider $[x,\mathfrak{g}_{-1}]=0$, $x \in \mathfrak{g}_k$, $k \ge 0$; then the space\\
\begin{equation*}
  \bigoplus_{k,l=0}^{\infty} (ad_{s}(\mathfrak{g}_1))^k (ad_{s}(\mathfrak{g}_0))^l (x)
\end{equation*}
    is obviously a nontrivial homogeneous hom-ideal of the hom-Lie superalgebrag. Hence $x=0$.
\end{proof}
Now, we consider an important notion of hom-Lie superalgebra.

\begin{definition}
	Let $\hat{\mathfrak{g}}$ be a $\mathds{Z}_2$-graded space, decomposed into a direct sum of $\mathds{Z}_2$-graded subspaces, $\hat{\mathfrak{g}} = \mathfrak{g}_{-1} \oplus \mathfrak{g}_0 \oplus \mathfrak{g}_1$. Suppose that whenever $|i+j| \le 1$, a bilinear operation is defined $\mathfrak{g}_i \times \mathfrak{g}_j \to \mathfrak{g}_{i+j}$, such that $((x,y) \mapsto [x,y])$, satisfying the axiom of anticommutativity and the hom-Jacobi identity for hom-Lie superalgebras, each time the three terms of the identity are defined. Then $\hat{\mathfrak{g}}$ is called a local hom-Lie superalgebra.
\end{definition}
From the above definition we have the following remark.
\begin{remark}
  For a $\mathds{Z}$-graded hom-Lie superalgebra $\mathfrak{g}= \bigoplus \mathfrak{g}_i$, there corresponds a local hom-Lie superalgebra $\mathfrak{g}_{-1} \oplus \mathfrak{g}_0 \oplus \mathfrak{g}_1$, which we call the local part of $\mathfrak{g}$.
\end{remark}

Homomorphisms, transivity and bitransivity, for a local hom-Lie superalgebra are defined as for a $\mathds{Z}$-graded hom-Lie superalgebra.
In this section we consider only $\mathds{Z}$-graded hom-Lie superalgebras $\mathfrak{g}= \bigoplus \mathfrak{g}_i$ in which the subspace $\mathfrak{g}_{-1} \oplus \mathfrak{g}_0 \oplus \mathfrak{g}_1$ generates $\mathfrak{g}$.\\

\begin{definition}
	Let $\hat{\mathfrak{g}} = \mathfrak{g}_{-1} \oplus \mathfrak{g}_0 \oplus \mathfrak{g}_1$ be a local hom-Lie superalgebra.
	\begin{itemize}
		\item 	A $\mathds{Z}$-graded hom-Lie superalgebra $\mathfrak{g}= \bigoplus \mathfrak{g}_i$ with local part $\hat{\mathfrak{g}}$ is said to be maximal ($\mathfrak{g}_{max}(\hat{\mathfrak{g}})$) if for any other $\mathds{Z}$-graded hom-Lie superalgebra $\mathfrak{g}'$, an isomorphism of the local part $\hat{ \mathfrak{g}}$ and $\hat{\mathfrak{g}}'$ extends to an epimorphism of $\mathfrak{g}$ onto $\mathfrak{g}'$.
		\item 	A $\mathds{Z}$-graded hom-Lie superalgebra $\mathfrak{g}= \bigoplus \mathfrak{g}_i$ with local part $\hat{\mathfrak{g}}$ is said to be minimal ($\mathfrak{g}_{min}(\hat{\mathfrak{g}})$) if for any other $\mathds{Z}$-graded hom-Lie superalgebra $\mathfrak{g}'$, an isomorphism of the local part $\hat{ \mathfrak{g}}$ and $\hat{\mathfrak{g}}'$ extends to an epimorphism of $\mathfrak{g}'$ onto $\mathfrak{g}$.
	\end{itemize}
\end{definition}

\begin{theorem}\label{Exist}
	Let $\hat{\mathfrak{g}} = \mathfrak{g}_{-1} \oplus \mathfrak{g}_0 \oplus \mathfrak{g}_1$ be a multiplicative local hom-Lie superalgebra such that $\alpha^2 = \alpha$. Then there exist a maximal and a minimal $\mathds{Z}$-graded hom-Lie superalgebra whose local parts are isomorphic to $\hat{\mathfrak{g}}$.
\end{theorem}

\begin{proof}
	Let $F (\mathfrak{\hat{g}})$ be the free hom-Lie superalgebra freely generated by $\mathfrak{\hat{g}}$ and $\tilde{I}$ be the hom-ideal of $F (\mathfrak{\hat{g}})$ generated by all relations of the form $[x,y]=z$ that hold in $\mathfrak{\hat{g}}$. Let $\mathfrak{\tilde{g}} = \frac{F (\mathfrak{\hat{g}})}{\tilde{I}}$ and denote by $\mathfrak{\tilde{g}}_{-1}$, $\mathfrak{\tilde{g}}_{0}$ and $\mathfrak{\tilde{g}}_{1}$ the images of the spaces $\mathfrak{g}_{-1}$, $\mathfrak{g}_{0}$ and $\mathfrak{g}_{1}$ respectively under the natural homomorphism of $F (\mathfrak{\hat{g}})$ onto $\mathfrak{\hat{g}}$.\\
Let $\mathfrak{\tilde{g}}_{-}$ [respectively, $\mathfrak{\tilde{g}}_{+}$ ] signify the hom-subalgebra of $\mathfrak{\tilde{g}}$ generated by $\mathfrak{\tilde{g}}_{-1}$ [respectively, $\mathfrak{\tilde{g}}_{1}$]. We claim that:
\begin{enumerate}
  \item[(\textbf{a})] $\mathfrak{\tilde{g}} = \mathfrak{\tilde{g}}_{-} \oplus \mathfrak{\tilde{g}}_{0} \oplus \mathfrak{\tilde{g}}_{+}$ and the hom-subalgebra $\mathfrak{\tilde{g}}_{-}$ and $\mathfrak{\tilde{g}}_{+}$ are freely generated by the spaces $\mathfrak{\tilde{g}}_{-1}$ and $\mathfrak{\tilde{g}}_{1}$, respectively.
  \item[(\textbf{b})] The local hom-Lie superalgebra $\mathfrak{\tilde{g}}_{-1} \oplus \mathfrak{\tilde{g}}_{0} \oplus \mathfrak{\tilde{g}}_{1}$ is isomorphic to $\mathfrak{\hat{g}}$.
\end{enumerate}
Let $T=T(\mathfrak{g}_{-1})= \oplus_{i=0}^{\infty}T_i$ be the tensor over the space $\mathfrak{g}_{-1}$. We define the representation $\phi$ of the hom-Lie superalgebra $F (\mathfrak{\hat{g}})$ as follows:\\
if $y \in \mathfrak{g}_{-1}$:
\begin{equation*}
  \varphi(y)a = y \otimes a, \quad a \in T;
\end{equation*}
if $z \in \mathfrak{g}_{0}$:
\begin{align*}
  & \varphi(z)1=0; \\
  & \varphi(z)a=[z,a], \quad a \in \mathfrak{g}_{-1} \subset T; \\
  & \varphi(z)(a_1 \otimes a_2)= \varphi (z)a_1 \otimes a_2 + a_1 \otimes \varphi(z)a_2, \quad a_1,a_2 \in T;
\end{align*}
if $x \in \mathfrak{g}_{1}$:
\begin{align*}
    & \varphi(x)1=0; \\
    & \varphi(x)(a_1 \otimes a_2)= \varphi ([x,\alpha(a_1)]) a_2 + \alpha(a_1) \otimes \varphi(x)a_2, \quad a_1 \in \mathfrak{g}_{-1} \subset T,\quad a_2 \in T;
\end{align*}
and $$\phi(\alpha (g)) = \varphi (\alpha(g)), \quad g \in F (\mathfrak{\hat{g}}).$$
We want to show that $\tilde{I} \subset Ker \phi$, so we are going to proof by induction on the degree of the elements of the $T$.\\
Let $x \in \mathfrak{g}_{1}$, $y \in \mathfrak{g}_{-1}$, $a_1 \in \mathfrak{g}_{-1} \subset T$, $z \in \mathfrak{g}_{0}$, $a_2 \in T$ and $a= a_1 \otimes a_2$. We have the following relations.
\begin{align*}
  (\varphi(\alpha(y)) \varphi(\alpha(z)) - \varphi(\alpha(z)) \varphi(\alpha(y)))a
  &= \varphi(\alpha(y)) \varphi(\alpha(z))a - \varphi(\alpha(z)) \varphi(\alpha(y))a\\
  &= \alpha(y) \otimes  \varphi(\alpha(z))a - \varphi(\alpha(z)) (\alpha(y) \otimes a)\\
  &= \alpha(y) \otimes  \varphi(\alpha(z))a - \varphi(\alpha(z)) \alpha(y) \otimes a\\
  & \quad - \alpha(y) \otimes \varphi(\alpha(z))a\\
  &= -\varphi(\alpha(z)) \alpha(y) \otimes a = - [\alpha(z),\alpha(y)] \otimes a\\
  &= [\alpha(y),\alpha(z)] \otimes a\\
  &= \varphi [\alpha(y),\alpha(z)]a.
\end{align*}
Then we have
\begin{align}\label{1}
  \varphi [\alpha(y),\alpha(z)] = \varphi(\alpha(y)) \varphi(\alpha(z)) - \varphi(\alpha(z)) \varphi(\alpha(y)).
\end{align}
Furthermore,
\begin{align*}
  (\varphi(\alpha(y)) \varphi(\alpha(x)) &- \varphi(\alpha(x)) \varphi(\alpha(y)))a =
  \varphi(\alpha(y)) \varphi(\alpha(x))a - \varphi(\alpha(x)) \varphi(\alpha(y))a\\
  &= \alpha(y) \otimes  \varphi(\alpha(x))a - \varphi(\alpha(x)) (\alpha(y) \otimes a)\\
  &= \alpha(y) \otimes  \varphi(\alpha(x))a - \varphi([\alpha(x),\alpha^2(y)])a - \alpha^2(y) \otimes \varphi(\alpha(x))a\\
  &= - \varphi([\alpha(x),\alpha(y)])a,
\end{align*}
thus
\begin{align}\label{2}
  \varphi [\alpha(x),\alpha(y)] = \varphi(\alpha(x)) \varphi(\alpha(y)) - \varphi(\alpha(y)) \varphi(\alpha(x)).
\end{align}
Finally we have
\begin{align*}
  \varphi(\alpha(z)) \varphi(\alpha(x))&(a_1 \otimes a_2) = \varphi(\alpha(z))(\varphi([\alpha(x),\alpha(a_1)])a_2 + \alpha(a_1) \otimes \varphi(\alpha(x))a_2 )\\
  &= \varphi(\alpha(z)) \varphi([\alpha(x),\alpha(a_1)])a_2 + \varphi(\alpha(z)) (\alpha(a_1) \otimes \varphi(\alpha(x))a_2)\\
  &= \varphi(\alpha(z)) \varphi([\alpha(x),\alpha(a_1)])a_2 + \varphi(\alpha(z))\alpha(a_1) \otimes \varphi(\alpha(x)a_2\\
  & \quad + \alpha(a_1) \otimes \varphi(\alpha(z)) \varphi(\alpha(x))a_2\\
  &= \varphi(\alpha(z)) \varphi([\alpha(x),\alpha(a_1)])a_2 + [\alpha(z),\alpha(a_1)] \otimes \varphi(\alpha(x)a_2\\
  & \quad + \alpha(a_1) \otimes \varphi(\alpha(z)) \varphi(\alpha(x))a_2,
\end{align*}
in the other hand we have
\begin{align*}
  \varphi(\alpha(x)) \varphi(\alpha(z))(a_1 \otimes a_2) &= \varphi(\alpha(x)) ( \varphi(\alpha(z))a_1 \otimes a_2 + a_1 \otimes \varphi(\alpha(z))a_2 )\\
  &= \varphi(\alpha(x)) ( \varphi(\alpha(z))a_1 \otimes a_2) + \varphi(\alpha(x)) (a_1 \otimes \varphi(\alpha(z))a_2)\\
  &= \varphi(\alpha(x)) ( [\alpha(z),a_1] \otimes a_2) + \varphi(\alpha(x)) (a_1 \otimes \varphi(\alpha(z))a_2)\\
  &= \varphi([\alpha(x),\alpha[\alpha(z),a_1]])a_2 + \alpha([\alpha(z),a_1]) \otimes \varphi(\alpha(x))a_2\\
  & \quad + \varphi([\alpha(x),\alpha(a_1)]) \varphi(\alpha(z))a_2 +  \alpha(a_1)\varphi(\alpha(x)) \varphi(\alpha(z))a_2\\
  &= \varphi([\alpha(x),[\alpha(z),\alpha(a_1)]])a_2 + [\alpha(z),\alpha(a_1)]) \otimes \varphi(\alpha(x))a_2\\
  & \quad + \varphi([\alpha(x),\alpha(a_1)]) \varphi(\alpha(z))a_2 +  \alpha(a_1)\varphi(\alpha(x)) \varphi(\alpha(z))a_2.
\end{align*}
By the induction hypothesis we have
\begin{align*}
  \varphi([x,z])a = (\varphi(x) \varphi(z) - \varphi(z) \varphi(x))a, \quad a \in T;
\end{align*}
so
\begin{align*}
  (\varphi(\alpha(z))\varphi(\alpha(x)) - & \varphi(\alpha(x)) \varphi(\alpha(z)) ) (a_1 \otimes a_2)\\
  & =(\varphi(\alpha(z))\varphi([\alpha(x),\alpha(a_1)])- \varphi ([\alpha(x),\alpha(a_1)])\varphi(\alpha(z)))a_2\\
  & + \alpha(a_1) \otimes ( \varphi(\alpha(z))\varphi(\alpha(x))- \varphi(\alpha(x)) \varphi(\alpha(z)) )a_2\\
  & - \varphi([\alpha(x),[\alpha(z),\alpha(a_1)]])a_2\\
  & = \varphi( [\alpha(z),[\alpha(x),\alpha(a_1)]] - [\alpha(x),[\alpha(z),\alpha(a_1)]] )a_2\\
  & \quad + \alpha(a_1) \otimes \varphi([\alpha(z),\alpha(x)])a_2\\
  &= \varphi( [\alpha(a_1),[\alpha(z),\alpha(x)]])a_2 + \alpha(a_1) \otimes \varphi([\alpha(z),\alpha(x)])a_2\\
  &= \varphi([[\alpha(z),\alpha(x)],\alpha(a_1)])+ \alpha(a_1) \otimes \varphi([\alpha(z),\alpha(x)])a_2\\
  &= \varphi([\alpha(z),\alpha(x)])(a_1 \otimes a_2);
\end{align*}
thus
\begin{align}\label{3}
  \varphi [\alpha(z),\alpha(x)] = \varphi(\alpha(z)) \varphi(\alpha(x)) - \varphi(\alpha(x)) \varphi(\alpha(z)).
\end{align}
Therefore by (\ref{1}), (\ref{2}) and (\ref{3}) we have proved $\tilde{I} \subset Ker \phi$.
Hence, the mapping $\phi$ induces a representation $\tilde{\phi}$ of the $\mathfrak{\tilde{g}}=\frac{F (\mathfrak{\hat{g}})}{\tilde{I}}$.\\
One can see that free hom-associative superalgebra $T' = \bigoplus_{i=1}^{\infty}T_i$ endowed with the structure
\begin{equation*}
  a \circ b = a \otimes b - (-1)^{|a||b|} b \otimes a, \quad a,b \in T',
\end{equation*}
is a hom-Lie superalgebra which is generated by $\mathfrak{g}_{-1}$ freely \cite{armakan2019enveloping}.\\
Since
\begin{equation*}
  \varphi([y_1,...,y_r])1 = y_1 \circ ... \circ y_r, \quad y_1,...,y_r \in \mathfrak{g}_{-1}.
\end{equation*}

Therefore $\mathfrak{\tilde{g}}_{-}$ is freely generated by $\mathfrak{\tilde{g}}_{-1}$ and the space $\mathfrak{g}_{-1}$ is mapped onto $\mathfrak{\tilde{g}}_{-1}$ isomorphically by natural homomorphism of $F (\mathfrak{\hat{g}})$ onto $\mathfrak{\tilde{g}}$. In the same way we can obtain a representation $\tilde{\phi}_1$ of the $\tilde{\mathfrak{g}}$ and we can see that $\mathfrak{\tilde{g}}_{+}$ is generated by $\mathfrak{\tilde{g}}_{1}$ freely and $\mathfrak{g}_{1}$ is mapped onto $\mathfrak{\tilde{g}}_{1}$ isomorphically.

Let $g_{-} \in \mathfrak{\tilde{g}}_{-}$, $g_{0} \in \mathfrak{\tilde{g}}_{0}$, $g_{+} \in \mathfrak{\tilde{g}}_{+}$ and $g = g_{-} + g_0 + g_{+} =0$ then $\phi(g)1= g_{-}=0$ and $\tilde{\phi}_1(g)1= g_{+}=0$; hence $g_0 =0$. Thus the sum of $\mathfrak{\tilde{g}}_{-}$, $\mathfrak{\tilde{g}}_{0}$ and $\mathfrak{\tilde{g}}_{+}$ is direct. So we have proved part $(\textbf{a})$.

The local hom-Lie superalgebra $\hat{\mathfrak{g}} = \mathfrak{g}_{-1} \oplus \mathfrak{g}_0 \oplus \mathfrak{g}_1$ embedding into a local hom-Lie superalgebra $\hat{\mathfrak{g}}' = \mathfrak{g}'_{-1} \oplus \mathfrak{g}_0 \oplus \mathfrak{g}_1$, where $\mathfrak{g}'_{-1} = \mathfrak{g}_{-1} \oplus V$ and $[V,\mathfrak{g}_1]=0$, $[\mathfrak{g}_0,V]\subset V$ and the representation of $\mathfrak{g}_0$ on $V$ is faithful, we can suppose that in the local hom-Lie superalgebra $\hat{\mathfrak{g}}$, the representation of $\mathfrak{g}_0$ on $\mathfrak{g}_{-1}$ is faithful. But $ Ker{\tilde{\phi}} \cap \mathfrak{g}_0 = 0$, therefore $\mathfrak{g}_0$ is mapped onto $\hat{\mathfrak{g}}_{0}$ isomorphically. So we have proved (\textbf{b}).

Hence hom-Lie superalgebra $\tilde{\mathfrak{g}} = \bigoplus \tilde{\mathfrak{g}}_i$ is a maximal $\mathds{Z}$-graded hom-Lie superalgebra with local part $\hat{\mathfrak{g}}$.

It is obvious that among the graded hom-ideals that intersection with $\mathfrak{\tilde{g}}_{-1} \oplus \mathfrak{\tilde{g}}_{0} \oplus \mathfrak{\tilde{g}}_{1}$ is zero, there is a unique maximal hom-ideal $I$. The hom-Lie superalgebra $\mathfrak{g}= \frac{\mathfrak{\tilde{g}}}{I}$ is a minimal $\mathds{Z}$-graded hom-Lie superalgebra with local part $\hat{\mathfrak{g}}$.
\end{proof}

\begin{corollary}
	Let $\hat{\mathfrak{g}} = \mathfrak{g}_{-1} \oplus \mathfrak{g}_0 \oplus \mathfrak{g}_1$ be a local hom-Lie superalgebra.
	\begin{enumerate}
		\item [i.] Any $\mathds{Z}$-graded hom-Lie superalgebra whose local part is isomorphic to $\hat{ \mathfrak{g}}$, is a quotient of maximal $\mathds{Z}$-graded hom-Lie superalgebra $\mathfrak{g}_{max}(\hat{\mathfrak{g}})$.
		\item [ii.] Minimal $\mathds{Z}$-graded hom-Lie superalgebra $\mathfrak{g}_{min}(\hat{\mathfrak{g}})$ is a quotient of any $\mathds{Z}$-graded hom-Lie superalgebra whose local part is isomorphic to $\hat{ \mathfrak{g}}$.
		\item [iii.] $\mathfrak{g}_{max}(\hat{\mathfrak{g}})$ has a unique maximal graded hom-ideal $J_{max}$ such that
		\begin{equation*}
			J_{max} \cap \hat{\mathfrak{g}} = \{0\} \quad
             \text{and} \quad
             \mathfrak{g}_{max}(\hat{\mathfrak{g}})/J_{max} = \mathfrak{g}_{min}(\hat{\mathfrak{g}}).
		\end{equation*}
	\end{enumerate}
\end{corollary}

\begin{proof}
	The proof is obvious by theorem \ref{Exist}.
\end{proof}

\begin{proposition}
	\begin{enumerate}
		\item [i.] A bitransitive $\mathds{Z}$-graded hom-Lie superalgebra is minimal.
		\item [ii.] A minimal $\mathds{Z}$-graded hom-Lie superalgebra with bitransitive local part is bitransitive.
		\item [iii.] Two bitransitive $\mathds{Z}$-graded hom-Lie superalgebra are isomorphic, if and only if their local parts are isomorphic.
	\end{enumerate}
\end{proposition}
\begin{proof}
	If the bitransitive $\mathds{Z}$-graded hom-Lie superalgebra $\mathfrak{g}= \bigoplus \mathfrak{g}_i$ were not minimal, there would exist a nonzero graded hom-ideal $J \subset \mathfrak{g}$ such that
	 \begin{equation*}
		J \cap (\mathfrak{g}_{-1} \oplus \mathfrak{g}_0 \oplus \mathfrak{g}_1) =0.
	 \end{equation*}
	Choose the least k in absolute value such that
	 \begin{equation*}
		[J\cap \mathfrak{g}_k , \mathfrak{g}_{-1}]=0,
	\end{equation*}
	\begin{equation*}
		J \cap \mathfrak{g}_k \ne 0.
	\end{equation*}
	Let $k>0$, then which contradicts the transivity of $\mathfrak{g}$. Thus we have proved (i).\\
	Assertion (ii) is proved in the same way as proposition \ref{prop1}.\\
	The last statement follws from (i).
\end{proof}

\section{Invariant bilinear forms}
In this section we want to study some property of invariant bilinear forms on $\mathds{Z}$-graded hom-Lie superalgebras. At first we recall some definitions on bilinear forms. After that we prove that a consistent supersymmetric $\alpha$-invariant bilinear form on the local part can be extended uniquely to a consistent supersymmetric $\alpha$-invariant bilinear form with the same property on the whole $\mathds{Z}$-graded hom-Lie superalgebra.
\begin{definition} \cite{benayadi2014hom}
	Let $(\mathfrak{g},[.,.],\alpha)$ be a hom-Lie superalgebras. Let $f$ be a bilinear form on $\mathfrak{g}$.\\
	The form $f$ is said to be consistent, if
	\begin{equation*}
	   f(x,y)=0 \quad \text{for all} \quad x \in g_0, y \in \mathfrak{g}_1;
	\end{equation*}
	$f$ is said to be nondegenerate, if
	\begin{equation*}
	\mathfrak{g}^{\perp}=\{x \in \mathfrak{g} | f(x,y)=0, \quad \text{for all} \quad y \in \mathfrak{g} \};
	\end{equation*}
	$f$ is said to be invariant, if
	\begin{equation*}
		f([x,y],z)=f(x,[y,z])  \quad \text{for all} \quad x, y, z \in \mathfrak{g};
	\end{equation*}
	$f$ is said to be supersymmetric, if
	\begin{equation*}
	   f(x,y)=-(-1)^{|x||y|}f(y,x) \quad \text{for all} \quad x, y, z \in \mathfrak{g}.
	\end{equation*}
    An invariant bilinear form $f$ is said to be $\alpha$-invariant, if
    \begin{equation*}
		f(\alpha(x),y)=f(x,\alpha(y)) \quad \text{for all} \quad x, y, z \in \mathfrak{g}.
	\end{equation*}
\end{definition}

It is obvious that the kernel of an invariant form $f$ on $\mathfrak{g}$ (that is, the set of $x \in \mathfrak{g}$ for which $f(x,\mathfrak{g})=0$) is an hom-ideal in $\mathfrak{g}$. Hence we have the following proposition.

\begin{proposition}
	Let $\mathfrak{g}$ be a simple hom-Lie superalgebra. Then every nonzero invariant form on $\mathfrak{g}$ is nondegenerate and any two invariant forms on $\mathfrak{g}$ are proportional.
\end{proposition}

\begin{proposition}
Let $(.,.)$ be a consistent supersymmetric $\alpha$-invariant bilinear form on the local part of a $\mathds{Z}$-graded hom-Lie superalgebra $\mathfrak{g}= \bigoplus \mathfrak{g}_i$ for which $(\mathfrak{g}_i,\mathfrak{g}_j)=0$, where $i+j \ne 0$. If $\mathfrak{g}_{-1} \oplus \mathfrak{g}_{0} \oplus \mathfrak{g}_{1}$ generates $\mathfrak{g}$, then the form can be extended uniquely to a consistent supersymmetric $\alpha$-invariant bilinear form with the same property on the whole $\mathfrak{g}$.
\end{proposition}

\begin{proof}
	Let $(\mathfrak{g}_i,\mathfrak{g}_j)=0$, where $i+j \ne 0$. We prove by induction on $k$ such that for $x \in \mathfrak{g}_i$, $y \in \mathfrak{g}_{-i}$, $|i| \le k$, we can define $(x,y)$ so that the bilinear form $(.,.)$ be $\alpha$-invariant so long as all the elements in this equation lie in the space $\bigoplus_{i=-k}^{k} \mathfrak{g}_i$. When $k=1$, the assertion holds by hypothesis, since $(.,.)$ is defined on $\mathfrak{g}_{-1} \oplus \mathfrak{g}_{0} \oplus \mathfrak{g}_{1}$. For $ i>0$, let $x_i$ and $y_i$ be elements in $\mathfrak{g}_i$ and $\mathfrak{g}_{-i}$ respectively. Suppose that
	\begin{equation*}
		A=(-1)^{|s-k||s|}(-1)^{|s-k||k|}([[x_{k-s},x_s],\alpha(y_{k-s})],y_s).
	\end{equation*}
	By the induction hypothesis we have the following equality ($0<s<k$):
		\begin{flalign*}
		A&=(-1)^{|s-k||s|}(-1)^{|s-k||k|}([[x_{k-s},x_s],\alpha(y_{k-s})],y_s)\\
        &=(-1)^{|k-s||s-k|}([\alpha(x_{k-s}),[x_s,y_{k-s}]],y_s) +(-1)^{|k-s||s|}([\alpha(x_s),[y_{k-s},x_{k-s}]],y_s)&\\
		&= (-1)^{|k-s||s-k|}([\alpha(x_{k-s}),[x_s,y_{k-s}]],y_s)
        - (-1)^{|k-s||s|}([[y_{k-s},x_{k-s}],\alpha(x_s)],y_s)&\\
		&= -(-1)^{|k-s||s-k|}(-1)^{|k-s||2s-k|} ([[x_s,y_{k-s}],\alpha(x_{k-s})],y_s)\\
	    & \quad   +(-1)^{|k-s||s|}(-1)^{|k-s||s-k|} ([[x_{k-s},y_{k-s}],\alpha(x_s)],y_s)&\\
		&=  -(-1)^{|k-s||s-k|}(-1)^{|k-s||2s-k|}  ([x_s,y_{k-s}],[\alpha(x_{k-s}),y_s])\\
        & \quad + (-1)^{|k-s||s|}(-1)^{|k-s||k-s|} ([x_{k-s},y_{k-s}],[\alpha(x_s),y_s])&\\
		&=  (-1)^{|k-s||s-k|}(-1)^{|k-s||2s-k|}(-1)^{|k-s||s|} ([x_s,y_{k-s}],[y_s,\alpha(x_{k-s})])\\
        & \quad + (-1)^{|k-s||s|}(-1)^{|k-s||k-s|} ([x_{k-s},y_{k-s}],[x_s,\alpha(y_s)])&\\
		&=  (-1)^{|k-s||s-k|}(-1)^{|k-s||2s-k|}(-1)^{|k-s||s|} ([x_s,y_{k-s}],[y_s,\alpha(x_{k-s})])\\
        & \quad + (-1)^{|k-s||s|}(-1)^{|k-s||k-s|}  ([x_{k-s},y_{k-s}],[x_s,\alpha(y_s)])&\\
		&= (-1)^{|k-s||s-k|}(-1)^{|k-s||2s-k|}(-1)^{|k-s||s|} (x_s,[y_{k-s},[y_s,\alpha(x_{k-s})]])\\
        & \quad - (-1)^{|k-s||s|}(-1)^{|k-s||k-s|} ([x_s,\alpha(y_s)],[x_{k-s},y_{k-s}])&\\
		&= (-1)^{|k-s||s-k|}(-1)^{|k-s||2s-k|}(-1)^{|k-s||s|} (x_s,[y_{k-s},[y_s,\alpha(x_{k-s})]])\\
        & \quad - (-1)^{|k-s||s|}(-1)^{|k-s||k-s|} (x_s,[\alpha(y_s),[x_{k-s},y_{k-s}]]),&
	\end{flalign*}
therefore we have
	\begin{flalign*}
		(-1)^{|s-k||k|}([[x_{k-s},x_s],& \alpha(y_{k-s})],y_s)\\
		&= (-1)^{|k-s||s-k|}(-1)^{|k-s||2s-k|} (x_s,[y_{k-s},[y_s,\alpha(x_{k-s})]])&\\
		&\quad - (-1)^{|k-s||k-s|} (x_s,[\alpha(y_s),[x_{k-s},y_{k-s}]],&
	\end{flalign*}
finally
	\begin{align*}
	([[x_{k-s},x_s],\alpha(y_{k-s})],y_s) =(x_s,[[\alpha(y_{k-s}),y_s],x_{k-s}]).
	\end{align*}
If we set
	\begin{align*}
		(\alpha[x_{k-s},x_s],[y_{k-s},y_s])
		= ([\alpha[x_{k-s},x_s],y_{k-s}],y_s)
		= ([[x_{k-s},x_s],\alpha(y_{k-s})],y_s),
	\end{align*}
then this definition will be well-defined and the form will satisfy the induction hypothesis.
\end{proof}

Now, we state some conditions for the simplicity of hom-Lie superalgebras.

\begin{proposition}
	Let $\mathfrak{g}=\mathfrak{g}_{\bar{0}} \oplus \mathfrak{g}_{\bar{1}}$ be a hom-Lie superalgebra which satisfies the following conditions:
	\begin{enumerate}
		\item [i.] The representation of $\mathfrak{g}_{\bar{0}}$ on $\mathfrak{g}_{\bar{1}}$ is faithfull and irreducible
		\item [ii.] $[\mathfrak{g}_{\bar{1}},\mathfrak{g}_{\bar{1}}]=\mathfrak{g}_{\bar{0}}$
	\end{enumerate}
	then the hom-Lie superalgebra $\mathfrak{g}$ is simple, provided that $\mathfrak{g}_{\bar{0}} \ne \{0\}$.
\end{proposition}
\begin{proof}
  The proof is straightforward.
\end{proof}

\begin{proposition}
	Let $\mathfrak{g}= \bigoplus_{i \ge -1} \mathfrak{g}_{i}$ be a $\mathds{Z}$-graded hom-Lie superalgebra such that $\mathfrak{g}_{-1} \ne \{0\}$.
	\begin{enumerate}
		\item [i.] If hom-Lie superalgebra $\mathfrak{g}$ is simple, then:
		\begin{enumerate}
			\item [1.] $\mathfrak{g}$ is transitive and irreducible.
			\item [2.] $[\mathfrak{g}_{-1},\mathfrak{g}_{1}]=\mathfrak{g}_{0}$.
			\item [3.] $\mathfrak{g}_{0} \ne \{0\}$, $\mathfrak{g}_{1} \ne \{0\}$.
		\end{enumerate}
		\item[ii.] Suppose that $\mathfrak{g}_{1} \ne \{0\}$ and $\mathfrak{g}$ be transitive and irreducible such that
		\begin{equation*}
		[\mathfrak{g}_{n},\mathfrak{g}_{1}]=\mathfrak{g}_{n+1} \quad n \ge -1,
		\end{equation*}
		then $\mathfrak{g}$ is simple.
	\end{enumerate}
\end{proposition}

\begin{proof}
	Let $V$ be a $\mathds{Z}_2$-graded subspace of $\mathfrak{g}$ such that:
	\begin{equation*}
	[\mathfrak{g}_{-1},V] \subset V, \qquad [\mathfrak{g}_{0},V] \subset V.
	\end{equation*}
	We set $\mathfrak{g}^{+}=\bigoplus_{n \ge 1} \mathfrak{g}_n$ and define for all positive integer $n \ge 0$:
	\begin{equation*}
	V^n=[\mathfrak{g}^{+},[\mathfrak{g}^{+},...,[\mathfrak{g}^{+},V]]], \qquad \text{n factors} \quad \mathfrak{g}^{+}
	\end{equation*}
	then we can see that
	\begin{equation*}
		\hat{V} = \Sigma_{n \ge 0} V^{n}
	\end{equation*}
	is graded hom-ideal of the hom-Lie superalgebra $\mathfrak{g}$.
	Suppose that $\mathfrak{g}$ is simple. If we set
	\begin{equation*}
	V= \{A \in \bigoplus_{n \ge 0} \mathfrak{g}_n \quad | [A,\mathfrak{g}_{-1}]=\{0 \}\}
	\end{equation*}
	then hom-ideal $\hat{V}$ is contained in $\bigoplus_{n \ge 0} \mathfrak{g}_n$.
	Since $\mathfrak{g}$ is simple, $\hat{V} = \{0\}$, so $V= \{0\}$ and therefore $\mathfrak{g}$ is transitive.\\
	Let $V$ be a nonzero graded $\mathfrak{g}_0$-submodule of $\mathfrak{g}_{-1}$. Then $\hat{V}$ is nonzero and is contained in $V \oplus \bigoplus_{n \ge 0} \mathfrak{g}_n$.
	Since $\mathfrak{g}$ is simple, implies that $\hat{V} = \mathfrak{g}$, thus $V=\mathfrak{g}_{-1}$ and so $\mathfrak{g}$ is irreducible.\\
	On the other hand, since $\mathfrak{g}_{-1} \oplus [\mathfrak{g}_{-1} , \mathfrak{g}_1 ] \oplus \mathfrak{g}^{+}$ is a graded hom-ideal of $\mathfrak{g}$ and $\mathfrak{g}$ is simple, so $[\mathfrak{g}_{-1} , \mathfrak{g}_1] = \mathfrak{g}_0$.\\
The last statement is obvious.
\end{proof}
Thus, in this paper we showed that there is a maximal (resp., minimal) $\mathds{Z}$-graded hom-Lie superalgebra for a given local hom-Lie superalgebra. Morever, we proved that a consistent supersymmetric $\alpha$-invariant form on the local part can be extended uniquely to a bilinear form with the same property on the whole $\mathds{Z}$-graded hom-Lie superalgebra. Furthermore, we found out the condition in which the $\mathds{Z}$-graded hom-Lie superalgebra is simple.
\section*{Acknowledgment}
The authors would like to thank Shiraz University, for financial support, which leads to the formation of this manuscript. This research is supported by grant no. 97grc1m82582 Shiraz University, Shiraz, Iran.
\section*{Competing interests}
The authors declare that they have no competing interests.

\end{document}